\documentclass[]{article}
\usepackage{amsmath,amsfonts,amssymb}
\usepackage{amsthm}
\usepackage{mathtools}
\usepackage{cases}
\usepackage[mathfrak]{}
\usepackage[latin1]{inputenc}
\usepackage[T1]{fontenc}
\usepackage{verbatim}
\usepackage{graphicx}
\usepackage{float}
\usepackage{mathrsfs}
\usepackage [all]{xy}
\usepackage{color}
\usepackage{tikz}
\usepackage{comment}
\newtheorem{thm}{Theorem}[section]
 \newtheorem{cor}[thm]{Corollary}
 \newtheorem{lem}[thm]{Lemma}
 \newtheorem{prop}[thm]{Proposition}
 \theoremstyle{definition}
 \newtheorem{defn}[thm]{Definition}
 \theoremstyle{remark}
 \newtheorem{rem}[thm]{Remark}
 \newtheorem{ex}{Example}
 \numberwithin{equation}{section}

\usepackage{authblk}
\title{The Chow ring of a sequence of point and rational curve blow-ups}
\author[1]{Daniel Camaz\'on-Portela \footnote{The author was partially supported by grant PID2022-138906NB-C21 funded by MICIU/AEI/ 10.13039/501100011033 and by ERDF A way of making Europe.}}
\affil[1]{Department of Mathematics, University of Almer\'ia, Carretera Sacramento, SN, Almer\'ia, 04120, Spain}
\date{}                     
\begin{document}
  \maketitle



\begin{abstract}
Given a sequence of point and rational curve blow-ups blow-ups of smooth $3-$dimensional projective varieties $Z_{i}$ defined over an algebraically closed field $\mathit{k}$, $Z_{s}\xrightarrow{\pi_{s}} Z_{s-1}\xrightarrow{\pi_{s-1}}\cdot\cdot\cdot\xrightarrow{\pi_{2}} Z_{1}\xrightarrow{\pi_{1}} Z_{0}$, with $Z_{0}\cong\mathbb{P}^{3}$, we give an explicit presentations of the Chow ring $A^{\bullet}(Z_{s})$ of its sky. We prove that, in contrast to the case of sequences of point blow-ups, the skies of two sequences of point and rational curve blow-ups of the same length and even with the same proximity relations may have non-isomorphic Chow rings. Moreover, we explore some necessary conditions for the existence of such an isomorphism under some proximity configurations, and we apply the previous results in order to establish the allowed proximity type between two irreducible components of the exceptional divisor when both are regularly and projectively contractable. 
\end{abstract}

\section{Introduction}

An important algebraic invariant of a projective algebraic variety $X$ is the Chow ring $A^{\bullet}(X)$ of algebraic cycles on $X$ modulo rational equivalence. It is graded by the codimension of cycles, that is, $A^{\bullet}(X)=\oplus_{i=0}^{dim(X)} A^{i}(X)$, where $A^{i}(X)$ denotes the Chow group of codimension $i$, and its ring structure comes from the intersection product. Moreover, $A^{\bullet}(X)$ is isomorphic to the homology ring for an important class of schemes (see \cite[Example 19.1.11]{Fulton98}). However, describing the product structure of $A^{\bullet}(X)$ is in general a challenging task. It is sometimes possible to describe the product structure in certain special cases where the cycle class map is no longer injective. For instance, there are well-known examples of Abelian varieties \cite{Beauville86}, $K3$ surfaces \cite{BeauvilleVoisin04}, Fano fourfolds of $K3$ type \cite{BolognesiLaterveer25}, Cynk-Hulek Calabi-Yau varieties and Schreieder varieties \cite{LaterveerVial20} and cubic hypersurfaces \cite{Anthony21}. Beyond these cases, it is natural to wonder to what extent one may determine the structure of the Chow ring for some other families of varieties.  

Morphisms between varieties induce maps between their Chow groups, and the study of such maps constitutes an active research area. In \cite{Aluffi05} Aluffi introduced the concept of modification system, that is, the collection of varieties mapping properly and birationally onto a fixed variety $X$, $\mathcal{C}_{X}$, and he defined its Chow group as the inverse limit of the Chow groups under push-forward. Then he proved that equivalent systems have isomorphic Chow groups, what allowed him to give a proof of the invariance of Chern classes. In \cite{Aluffi06} the same author defined an enriched notion of Chow groups for algebraic varieties, agreeing with the conventional notion for complete varieties, but enjoying a functorial push-forward for arbitrary maps. This tool allowed him to glue intersection theoretic information across elements of a stratification of a variety, and he gave a direct construction of Chern-Schwartz-MacPherson classes of singular varieties.

In this context, morphisms obtained as composition of sequences of blow-ups play an essential role as shown in the work of Encinas and Villamayor \cite{EncinasVillamayor00} where, by taking the foundational work of Hironaka \cite{Hironaka64} as starting point, they studied a constructive proof of desingularization as the outcome of a process obtained by successively blowing-up the maximum stratum of a function. This result lead some authors to study the intersection theory of the out-coming resolution varieties, but mostly focused on the numerical information coming from the degree of the zero cycles. In \cite{CampilloReguera94} Campillo and Reguera studied morphisms given by composition of a sequence of point blow-ups of smooth $d-$dimensional varieties in terms of combinatorial information coming from the $d-$ary intersection form on divisors with exceptional support and defined its associated weighted polyhedron. In \cite{Tsuchihashi86} Tsuchihashi defined a weighted dual graph of a toric divisor arising as the exceptional set of a resolution of a $3-$dimensional cusp singularity as its dual graph with a pair of integers attached, corresponding to the self-intersection numbers of the double curve defined by the scheme theoretic intersection of the corresponding irreducible components. Moreover, he proved that an arbitrary weighted graph on a compact topological surface is a weighted dual graph of a toric divisor arising as the exceptional set if and only if it satisfies the monodromy condition and the convexity condition.


The aim of this work is to fully describe an algebraic object naturally attached to the sky of a sequence of point and rational curve blow-ups, that is, its Chow ring $A^{\bullet}(Z_{s})$. This is a natural continuation of our previous work \cite{Camazon24} where we described the Chow ring of the sky of a sequence of point blow-ups. Notice that whereas in \cite{CampilloReguera94} and \cite{Tsuchihashi86} the labels appearing in the weighted polyhedron (resp. weighted graph) are obtained from the $d-$ary (resp. $3-$ary) intersection form, that is, they just contained information about $A_{0}(Z_{s})$, we give an explicit description of the whole ring as a $\mathbb{Z}-$algebra by providing its generators and relations. The article is structured as follows. The first section of this work is devoted to introduce some basic concepts as well as to recall some well-known results about intersection theory and blow-ups. For example, we revisit those describing the Chow ring of the exceptional components as well as some others concerning the relations between the Chow rings of two varieties, one of them obtained as the blow-up of the other one. In the next section, that contains the main results of this article, we describe explicitly the generators of the Chow ring $A^{\bullet}(Z_{s})$ of the sky of a sequence of point and rational curve blow-ups with ground variety $Z_{0}\cong\mathbb{P}^{3}$ and their corresponding relations. Moreover, throughout some clarifying examples, we explore some necessary conditions for the existence of such an isomorphism under some proximity configurations, and we apply the previous results in order to establish the allowed proximity type between two irreducible components of the exceptional divisor when both are regularly and projectively contractable. 

\section{Preliminaries} 
\subsection{Sequences of blow-ups at smooth centers}

Through this work will restrict ourselves to the case of sequences of point and rational curve blow-ups over an algebraically closed field $\mathit{k}$ with ground variety isomorphic to the projective space, that is, sequences $Z_{s}\xrightarrow{\pi_{s}} Z_{s-1}\xrightarrow{\pi_{s-1}}\ldots\xrightarrow{\pi_{2}} Z_{1}\xrightarrow{\pi_{1}} Z_{0}$, where either centers are closed points, $C_{\alpha}=P$, or rational curves, $C_{\alpha}=\mathcal{C}$, and $Z_{0}\cong\mathbb{P}^{3}$. Firstly, we recall some fundamental and general concepts about sequences of blow-ups.

\begin{defn}\label{DefSeqBU}
A sequence of blow-ups over $\mathit{k}$ is defined as as a sequence of morphisms
\begin{equation*}
Z_{s}\xrightarrow{\pi_{s}} Z_{s-1}\xrightarrow{\pi_{s-1}}\cdot\cdot\cdot\xrightarrow{\pi_{2}} Z_{1}\xrightarrow{\pi_{1}} Z_{0},
\end{equation*}
between smooth irreducible $d-$dimensional varieties such that for $i\in\left\{0,\ldots,s-1\right\}$:
\begin{enumerate}
\item $\pi_{i+1}$ is the blow up of $Z_{i}$ at a smooth subvariety $C_{i+1}\xhookrightarrow{} Z_{i}$,
\item $codim(C_{i+1})\geq 2$,
\item if we denote by $E_{j}^{j}\xhookrightarrow{} Z_{j}$ the exceptional divisor of $\pi_{j}$, and for $k>j$ we denote by $E_{j}^{k}$ the strict transform of $E_{j}^{j}$ in $Z_{k}$, then $C_{i+1}$ has normal crossings with $\{E_{1}^{i}, E_{2}^{i},...,E_{i}^{i}\}$.
\end{enumerate}
\end{defn}

We denote by $\pi$ the composition $\pi_{1}\circ\pi_{2}\circ...\circ\pi_{s-1}\circ\pi_{s}$. 

\begin{defn}\label{DefSM}
A morphism $\pi: Z_{s}\rightarrow Z_{0}$ which can be expressed, in at least one way, as a composition of blow-ups with the conditions in Definition \ref{DefSeqBU} will be called a sequential morphism.
\end{defn}

\begin{rem}\label{NoPi}
Given a sequence of blow-ups $(Z_{s},...,Z_{0},\pi)$, we denote by $\pi_{s,i}: Z_{s}\rightarrow Z_{i}$ where $\pi_{s,i}=\pi_{i+1}\circ\pi_{i+2}\circ...\circ\pi_{s-1}\circ\pi_{s}$. 
\end{rem}

\begin{rem}
We will refer to $Z_{0}$ and $Z_{s}$ as the ground and the sky of the sequential morphism $\pi: Z_{s}\rightarrow Z_{0}$ respectively. Moreover we will denote by $E_{\beta}$ the irreducible components over $\mathit{k}$ of the exceptional divisor $E$ of $\pi$, that is we have $E=\bigcup_{\beta} E_{\beta}$.
\end{rem}

The centers $C_{i}$, in general, can have any dimension. We extend the well-known notion of proximity for point blow-ups.

\begin{defn}\label{DefGenProxC}
Given a sequence of blow-ups $(Z_{s}, ..., Z_{0},\pi)$ as in Definition \ref{DefSeqBU}, we say that $C_{j}$ is proximate (resp. $t-$proximate) to $C_{i}$, and write $C_{j}\xrightarrow{} C_{i}$ (resp. $C_{j}\xrightarrow{t} C_{i}$) if $C_{j}\subset E_{i}^{j-1}$ (resp. $C_{j}\cap E_{i}^{j-1}\neq\emptyset$ but $C_{j}\not\subset E_{i}^{j-1}$).
\end{defn}

Note that, if $C_{j}$ is either proximate or $t-$proximate to $C_{i}$ then $j>i$.

\begin{defn}\label{DefGenProxE}
Given a sequential morphism $\pi: Z\rightarrow Z_{0}$ and two irreducible exceptional components $E_{i}, E_{j}\subset E$, then we will say that $E_{j}$ is proximate (resp. $t-$proximate) to $E_{i}$ if there exists a sequence of blow-ups $(Z_{s},...,Z_{0},\pi)$ realizing the sequential morphism $\pi: Z\rightarrow Z_{0}$, such that $C_{j}$ is proximate (resp. $t-$proximate) to $C_{i}$.
\end{defn}

\subsection{Intersection theory of blow-ups}

In this subsection we recall some fundamental results of the intersection theory of blow-ups. 
\begin{rem}
Let $V$ be a $m-$dimensional irreducible subvariety of $X$. Although it is common to denote by $\left[V\right]$ to the equivalence class of $V$ in the Chow ring $A^{\bullet}(X)$, for simplicity we will denote also by $v$ to the equivalence class whenever there is not possible confusion.
\end{rem}
To begin with, we revisit the Chow ring of a projective space $\mathbb{P}^{n}$.
\begin{thm}\cite[Theorem 2.1.]{EisenbudHarris16}\label{ThmPn}
The Chow ring of $\mathbb{P}^{n}$ is 
\begin{equation*}
A^{\bullet}(\mathbb{P}^{n})=\mathbb{Z}\left[\varsigma\right]/(\varsigma^{n+1}),
\end{equation*}
where $\varsigma\in A^{1}(\mathbb{P}^{n})$ is the rational equivalence class of a hyperplane; more generally, the
class of a variety of codimension $k$ and degree $d$ is $d\varsigma^{k}$.
\end{thm}
As a consequence of Theorem \ref{ThmPn}, we have that the Chow ring of the ground variety $A^{\bullet}(Z_{0})$ is isomorphic to
\begin{equation*}
A^{\bullet}(Z_{0})\cong\mathbb{Z}\left[u\right]/(u^{4}), 
\end{equation*}
by sending $u$ to $h$, where $h\in A^{1}(Z_{0})$ is the rational equivalence class of any hyperplane $\left[H\right]$ in $\mathbb{P}^{3}$.

The exceptional divisor obtained by blowing-up at a smooth center $C_{\alpha}\subset Z_{\alpha-1}$ has a projective bundle structure over this one, being the following well-known result an explicit description of the Chow ring of this special type of projective varieties.
\begin{thm}\cite[Theorem 9.6.]{EisenbudHarris16}\label{ThmCRPB}
Let $V$ be a vector bundle of rank $r+1$ on a smooth projective variety $X$, and let $\varsigma=c_{1}(\mathcal{O}_{P(V)}(1))\in A^{1}(P(V))$, and $p: P(V)\rightarrow X$ the projection of the induced projective bundle. The map $p^{*}: A(X)\rightarrow A(P(V))$ is an injective ring homomorphism, and via this map one has the isomorphism of $A(X)$-algebras given by
\begin{equation*}
A(P(V))\cong A(X)\left[\varsigma\right]/(\varsigma^{r+1}+c_{1}(V)\varsigma^{r}+\cdot\cdot\cdot+c_{r+1}(V))
\end{equation*}
In particular, the group homomorphism $A(X)^{\oplus r+1}\rightarrow A(P(V))$ / given by $(\alpha_{0},...,\alpha_{r})\mapsto\sum\varsigma^{i}p^{*}(\alpha_{i})$ is an isomorphism, so that
\begin{equation*}
A(P(V))\cong\bigoplus_{i=0}^{r}\varsigma^{i}A(X)
\end{equation*}
as groups.
\end{thm}
Since for all $\alpha$ it is satisfied that either $E_{\alpha}^{\alpha}\cong\mathbb{P}^{2}$ or $E_{\alpha}^{\alpha}\cong\mathbb{F}_{\delta}$, where by $\mathbb{F}_{\delta}$ we denote a Hirzebruch surface of degree $\delta$ , then it follows from Theorem \ref{ThmPn} and Theorem \ref{ThmCRPB} that:
$$
A^{\bullet}(E_{\alpha}^{\alpha})\cong\begin{cases}
\mathbb{Z}\left[s\right]/(s^{3}) & \text{by sending}\enspace s\enspace\text{to}\enspace \varsigma_{\alpha}\enspace\text{if}\enspace C_{\alpha}=P_{\alpha}, \\ 
\mathbb{Z}\left[t,u\right]/(t^{2}+c_{1}(N_{\mathcal{C}_{\alpha}/Z_{\alpha-1}})t\cdot u, u^{2}) & \text{by sending}\enspace t, u\enspace\text{to}\enspace \varsigma_{\alpha}\enspace\text{and}\enspace w \\
& \text{ respectively if}\enspace C_{\alpha}=\mathcal{C}_{\alpha}, 
\end{cases}
$$
where $\varsigma_{\alpha}\in A^{1}(E_{\alpha}^{\alpha})$ is the rational class of any hyperplane and $w\in A^{1}(E_{\alpha}^{\alpha})$ corresponds to the rational class of a generic fiber. 

In the case where $E_{\alpha}^{\alpha}\cong\mathbb{F}_{\delta}$, the following result characterizes the classes of the irreducible non-singular rational curves on it. First we recall that given a Hirzebruch surface $X\cong\mathbb{F}_{\delta}$ there just exists a minimal section $S_{0}\subset X$ that verifies $s_{0}\cdot s_{0}=-\delta$.
\begin{prop}\cite[Proposition 5.3.5.]{Camazon25}\label{ProNSingRatCOnHir}
Given a Hirzebruch surface $\mathbb{F}_{\delta}$, then any irreducible non-singular rational curve $C\subset\mathbb{F}_{\delta}$ is of one of the following types:
\begin{enumerate}
\item either a section of class $s_{0}+bf$ with $b=0$ or $b\geq\delta$,
\item or a fiber $f$,
\item or a curve of class $2s_{0}+2f$ if $\delta=1$,
\item or a curve of class $as_{0}+f$ with $a>0$ if $\delta=0$.
\end{enumerate}
\end{prop}

Once the Chow ring of the exceptional components $E_{\alpha}^{\alpha}$ and the ground variety $Z_{0}$ is described, then the following question naturally arises: What is the relation between the Chow rings of the varieties $Z_{\alpha}$ and $Z_{\alpha+1}$?
\begin{equation}\label{BUDiag}
\xymatrix{ E_{\alpha+1}^{\alpha+1}\ar[r]^{j_{\alpha+1}}\ar[d]_{g_{\alpha+1}} & Z_{\alpha+1}\ar[d]^{\pi_{\alpha+1}} \\
C_{\alpha+1}\ar[r]^{i_{\alpha+1}} & Z_{\alpha} }
\end{equation}
Before answering it, we recall the following theorem that provides us with a pull-back formula under blow-ups.
\begin{thm}\cite[Theorem 6.7]{Fulton98}\label{ThmBUF}
(Blow-up Formula). Let $V$ be a $k-$dimensional subvariety of $Y$, and let $\widetilde{V}\subset\widetilde{Y}$ be the proper transform of $V$, i.e. the blow-up of $V$ along $V\cap X$.
Then
\begin{equation}
\pi^{*}v=\widetilde{v}+j_{*}\left\{c(\mathcal{Q})\cap g^{*}s(V\cap X, V)\right\}_{k}
\end{equation}
in $A_{k}\widetilde{Y}$, where $\mathcal{Q}=g^{*}\left(\frac{N_{X/Y}}{\mathcal{O}_{E}(-1)}\right)$.
\end{thm}

\begin{rem}
For $j>i$ we denote by $E_{i}^{j*}$ the total transform of $E_{i}^{i}$ by the morphism $\pi_{j,i}: Z_{j}\rightarrow Z_{i}$. By an abuse of notation $E_{i}^{i*}=E_{i}^{i}$. Note that by definition of the total transform and Theorem \ref{ThmBUF}, we have
\begin{equation*}
e_{i}^{k*}=e_{i}^{k}+\sum_{j>i} p_{ij}e_{j}^{k*}
\end{equation*}
where $p_{ij}=1$ if $i<j\leq k$ and $C_{j}$ is proximate to $C_{i}$ and $p_{ij}=0$ in any other case. \\
\end{rem}

The two following propositions give us an answer on how the Chow rings $A^{\bullet}(Z_{\alpha})$ and $A^{\bullet}(Z_{\alpha+1})$ are related by establishing the generators of the Chow groups $\left\{A_{k}(Z_{\alpha+1})\right\}_{k}$ as well as its multiplication rules. 
\begin{prop}\cite[Proposition 6.7.]{Fulton98}\label{ProGenCRBU}
\begin{enumerate}
\item (Key Formula). For all $x\in A_{k} C_{\alpha+1}$,
\begin{equation*}
\pi_{\alpha+1}^{*}i_{\alpha+1 *}(x)=j_{\alpha+1 *}(c_{d-1}(\mathcal{Q})\cap g_{\alpha+1}^{*}x)
\end{equation*}
in $A_{k} Z_{\alpha+1}$.
\item For all $y\in A_{k} Z_{\alpha}$, $\pi_{\alpha+1 *}\pi_{\alpha+1}^{*}y=y$ in $A_{k} Z_{\alpha}$.
\item If $\widetilde{x}\in A_{k} E_{\alpha+1}^{\alpha+1}$, and $g_{\alpha+1 *}\widetilde{x}=j_{\alpha+1}^{*}j_{\alpha+1 *}\widetilde{x}=0$, then $x=0$.
\item If $\widetilde{y}\in A_{k} Z_{\alpha+1}$, and $\pi_{\alpha+1 *}\widetilde{y}=j_{\alpha+1}^{*}\widetilde{y}=0$, then $y=0$.
\item There are split exact sequences
\begin{equation*}
0\xrightarrow{} A_{k} C_{\alpha+1}\xrightarrow{l} A_{k} E_{\alpha+1}^{\alpha+1}\oplus A_{k} Z_{\alpha}\xrightarrow{m} A_{k} Z_{\alpha+1}\xrightarrow{} 0
\end{equation*}
with $l(x)=(c_{d-1}(\mathcal{Q})\cap g_{\alpha+1}^{*}x, -i_{\alpha+1 *}x)$, and $m(\widetilde{x},y)=j_{\alpha+1 *}\widetilde{x}+\pi_{\alpha+1}^{*}y$. A left inverse for
$l$ is given by $(\widetilde{x}, y)\xrightarrow{}  g_{\alpha+1 *}(\widetilde{x})$.
\end{enumerate}
\end{prop}
\begin{prop}\cite[Example 8.3.9.]{Fulton98}\label{ProMRBU}
Let consider the blow-up diagram \ref{BUDiag}, with $Z_{\alpha}$, $C_{\alpha+1}$, and therefore $Z_{\alpha+1}$, $E_{\alpha+1}^{\alpha+1}$ non singular. The ring structure on $A^{\bullet}(Z_{\alpha+1})$ is determined by the following rules:
\begin{enumerate}
\item $\pi_{\alpha+1}^{*}y\cdot \pi_{\alpha+1}^{*}y^{'}=\pi_{\alpha+1}^{*}(y\cdot y^{'})$.
\item $j_{\alpha+1 *}(\widetilde{x})\cdot j_{\alpha+1 *}(\widetilde{x}^{'})=j_{\alpha+1 *}(c_{1}(j^{\alpha+1 *}\mathcal{O}_{Z_{\alpha+1}}(E_{\alpha+1}^{\alpha+1}))\cdot\widetilde{x}\cdot\widetilde{x}^{'})$.
\item $\pi_{\alpha+1}^{*}(y)\cdot j_{\alpha+1 *}(\widetilde{x})=j_{\alpha+1 *}((g_{\alpha+1}^{*}i^{\alpha+1 *}y)\cdot\widetilde{x})$.
\end{enumerate}
\end{prop}

\section{Main results}

By considering sequences of point and rational curve blow-ups with ground $Z_{0}\cong\mathbb{P}^{3}$, we are able to give generators of the Chow ring of the sky $A^{\bullet}(Z_{s})$ as a $\mathbb{Z}-$algebra. To begin with, let us consider the following partition of the centers of the sequence of blow-ups:
\begin{equation*} 
\left\{C_{i}\right\}_{i=1}^{s}=\left\{C_{i}\right\}_{i\in\mathcal{I}_{1}}\sqcup\left\{C_{i}\right\}_{i\in\mathcal{I}_{2}},
\end{equation*}
where $i\in\mathcal{I}_{1}$ if $dim(C_{i})=0$ and $i\in\mathcal{I}_{2}$ otherwise.

\begin{lem}\label{LemGCRPandCBU}
The Chow ring of the sky of the sequence $A^{\bullet}(Z_{s})$ is generated by $\left\{h^{s*},\left\{e_{\alpha}^{s*}\right\}_{\alpha\in\mathcal{I}_{1}},\left\{e_{\alpha}^{s*},w_{\alpha}^{s*}\right\}_{\alpha\in\mathcal{I}_{2}}\right\}$ as a $\mathbb{Z}-$algebra.
\end{lem}

\begin{proof}
The result follows by induction on $\alpha$. It is clear that $A^{\bullet}(Z_{0})$ is generated by $\left\{h\right\}$. Let us suppose that $A^{\bullet}(Z_{\alpha})$ is generated by 
\begin{equation*}
\left\{h^{\alpha*},\left\{e_{i}^{\alpha*}\right\}_{\substack{i\in\mathcal{I}_{1} \\ i\leq\alpha}},\left\{e_{i}^{\alpha*},w_{i}^{\beta*}\right\}_{\substack{i\in\mathcal{I}_{2} \\ i\leq\alpha}}\right\}.
\end{equation*}
Now we have to consider the two following settings: either $dim(C_{\alpha+1})=0$ or $dim(C_{\alpha+1})=1$. In the former case, since $E_{\alpha+1}^{\alpha+1}\cong\mathbb{P}^{2}$, that is $A^{\bullet}(E_{\alpha+1}^{\alpha+1})\cong\mathbb{Z}\left[s\right]/(s^{3})$, and $e_{\alpha+1}^{\alpha+1*}\cdot e_{\alpha+1}^{\alpha+1*}=-j_{\alpha+1*}(\varsigma_{\alpha+1})$ by Proposition \ref{ProMRBU}, then by Proposition \ref{ProGenCRBU} and Theorem \ref{ThmPn} we have that $A^{\bullet}(Z_{\alpha+1})$ is generated by
\begin{equation*} 
\left\{h^{\alpha+1*},\left\{e_{i}^{\alpha+1*}\right\}_{\substack{i\in\mathcal{I}_{1} \\ i\leq\alpha+1}},\left\{e_{i}^{\alpha+1*},w_{i}^{\alpha+1*}\right\}_{\substack{i\in\mathcal{I}_{2} \\ i\leq\alpha}}\right\}
\end{equation*}
 as a $\mathbb{Z}-$algebra. In the latter case, that is $dim(C_{\alpha+1})=1$, we have that $E_{\alpha+1}^{\alpha+1}\cong\mathbb{F}_{\delta}$, that is $A^{\bullet}(E_{\alpha+1}^{\alpha+1})\cong\mathbb{Z}\left[t,u\right]/(t^{2}+c_{1}(N_{\mathcal{C}_{\alpha}/Z_{\alpha-1}})t\cdot u, u^{2})$, and $e_{\alpha+1}^{\alpha+1*}\cdot e_{\alpha+1}^{\alpha+1*}=-j_{\alpha+1*}(\varsigma_{\alpha+1})$ by Proposition \ref{ProMRBU}. Then as a consequence of Proposition \ref{ProGenCRBU} and Theorem \ref{ThmCRPB} we have that $A^{\bullet}(Z_{\alpha+1})$ is generated by 
\begin{equation*}
\left\{h^{\alpha+1*},\left\{e_{i}^{\alpha+1*}\right\}_{\substack{i\in\mathcal{I}_{1} \\ i\leq\alpha}},\left\{e_{i}^{\alpha+1*},w_{i}^{\alpha+1*}\right\}_{\substack{i\in\mathcal{I}_{2} \\ i\leq\alpha+1}}\right\}.
\end{equation*}
\end{proof}

Now, in order to compute the relations between the generators, let us restrict firstly to the blow-up at the $\alpha+1-$level, that is $\pi_{\alpha+1}: Z_{\alpha+1}\rightarrow Z_{\alpha}$, where $C_{\alpha+1}$ is a rational curve, and $Z_{\alpha}$ is the sky of a sequence of $\#\mathcal{I}_{1}^{\alpha}$ point blow-ups and $\#\mathcal{I}_{2}^{\alpha}$ rational curve blow-ups. \\

From Proposition \ref{ProGenCRBU}  we know that $A^{\bullet}(Z_{\alpha+1})$ is generated by $\pi_{\alpha+1}^{*} A^{\bullet}(Z_{\alpha})$ and $j_{\alpha+1 *} A^{\bullet} E_{\alpha+1}^{\alpha+1}$. We can thus define a ring homomorphism
\begin{equation*}
f_{\alpha+1}: A^{\bullet}(Z_{\alpha})\left[e_{\alpha+1}^{\alpha+1},w_{\alpha+1}^{\alpha+1}\right]\rightarrow A^{\bullet}(Z_{\alpha+1}),
\end{equation*}
such that,
\begin{numcases}{f_{\alpha+1}(x)=}
\pi_{\alpha+1}^{*}(x) & if\enspace $x\in A^{\bullet}(Z_{\alpha})$, \\
j_{\alpha+1 *}(1) &  if\enspace $x=e_{\alpha+1}^{\alpha+1}$, \\
j_{\alpha+1*}(g_{\alpha+1}^{*}(P)) & if\enspace $x=w_{\alpha+1}^{\alpha+1}$,
\end{numcases}
where the class of $P$, $\left[P\right]\in A^{1}(C_{\alpha+1})$ is a generator of $A^{1}(C_{\alpha+1})$. Consequently, we have that
\begin{equation*}
A^{\bullet}(Z_{\alpha+1})\cong A^{\bullet}(Z_{\alpha})\left[e_{\alpha+1}^{\alpha+1},w_{\alpha+1}^{\alpha+1}\right]/ker f_{\alpha+1}.
\end{equation*}

\begin{thm}\label{ThmCRalpha+1}
The Chow ring of $Z_{\alpha+1}$, $A^{\bullet}(Z_{\alpha+1})$, is isomorphic to
\begin{equation*}
A^{\bullet}(Z_{\alpha+1})\cong\frac{A^{\bullet}(Z_{\alpha})\left[e_{\alpha+1}^{\alpha+1},w_{\alpha+1}^{\alpha+1}\right]}{\mathcal{J}_{\alpha+1}},
\end{equation*}
where
\begin{multline*}
\mathcal{J}_{\alpha+1}=\left(ker\: i_{\alpha+1}^{*}\cdot e_{\alpha+1}^{\alpha+1},h^{\alpha+1*}\cdot e_{\alpha+1}^{\alpha+1}-\mu_{0}w_{\alpha+1}^{\alpha+1},\left\{e_{\beta}^{\alpha*}\cdot e_{\alpha+1}^{\alpha+1}-\mu_{\beta}w_{\alpha+1}^{\alpha+1}\right\},(w_{\alpha+1}^{\alpha+1})^{2},\right. \\ h^{\alpha+1*}\cdot w_{\alpha+1}^{\alpha+1}, \left\{e_{\beta}^{\alpha*}\cdot w_{\alpha+1}^{\alpha+1}\right\}_{\beta=1}^{\alpha}, (e_{\alpha+1}^{\alpha+1})^{2}-c_{1}(N_{C_{\alpha+1}/Z_{\alpha}})w_{\alpha+1}^{\alpha+1}+\left[C_{\alpha+1}\right], \\ \left. e_{\alpha+1}^{\alpha+1}\cdot w_{\alpha+1}^{\alpha+1}+(h^{\alpha+1*})^{3}\right),
\end{multline*}
with $\mu_{\beta}=e_{\beta}^{\alpha*}\cdot \left[C_{\alpha+1}\right]$.
\end{thm} 

\begin{proof}
In order to compute the relations between the generators of $A^{\bullet}(Z_{\alpha+1})$, let us now consider the ring homomorphism induced by the inclusion $i_{\alpha+1}: C_{\alpha+1}\rightarrow Z_{\alpha}$, that is
\begin{equation*}
i_{\alpha+1}^{*}: A^{\bullet}(Z_{\alpha})\rightarrow A^{\bullet}(C_{\alpha+1}),
\end{equation*}
and let us denote by $\mathcal{I}_{\alpha+1}$ to
\begin{multline*}
\mathcal{I}_{\alpha+1}:=\left(\left\{a_{0}h^{\alpha*}+\sum_{\beta=1}^{\alpha}a_{\beta}e_{\beta}^{\alpha*}\right\},(h^{\alpha*})^{2},\left\{h^{\alpha*}\cdot e_{\beta}^{\alpha*}\right\}_{\beta=1}^{\alpha}, \right. \\ \left. \left\{e_{\beta}^{\alpha*}\cdot e_{\delta}^{\alpha*}\right\}_{\substack{\beta,\delta<\alpha+1 \\ \beta\neq\delta}},\left\{w_{\beta}^{\alpha*}\right\}_{\substack{\beta\in\mathcal{I}_{2} \\ \beta<\alpha+1}}\right),
\end{multline*}
where $\left\{a_{0}\mu_{0}+\sum_{\beta=1}^{\alpha}a_{\beta}\mu_{\beta}\right\}$ denotes the minimum set of relations of the finitely generated free abelian group $\mathcal{S}_{\alpha+1}$ generated by 
\begin{equation*}
\left\{\mu_{0}=deg(i_{\alpha+1}^{*}h^{\alpha*}), \left\{\mu_{i}=deg(i_{\alpha+1}^{*}e_{\beta}^{\alpha*})\right\}_{\beta=1}^{\alpha}\right\}.
\end{equation*}
Then, it can be proved that $Ker(i_{\alpha+1}^{*})=\mathcal{I}_{\alpha+1}$. Firstly, we will prove that $\mathcal{I}_{\alpha+1}\subset Ker(i_{\alpha+1}^{*})$. Since $C_{\alpha+1}$ is a rational curve, then $A^{1}(C_{\alpha+1})$ is generated by the class $\left[P\right]\in A^{1}(C_{\alpha+1})$ so we can conclude that $\left\{a_{0}m_{0}+\sum_{\beta=1}^{\alpha}a_{\beta}m_{\beta}\right\}\subset Ker(i_{\alpha+1}^{*})$. Moreover, we know that the pull-back morphism $i_{\alpha+1}^{*}: A^{\bullet}(Z_{\alpha})\rightarrow A^{\bullet}(C_{\alpha+1})$ is graded on codimension, so it follows that 
\begin{equation*}
\left((h^{\alpha*})^{2},\left\{h^{\alpha*}\cdot e_{\beta}^{\alpha*}\right\}_{\beta=1}^{\alpha},\left\{e_{\beta}^{\alpha*}\cdot e_{\delta}^{\alpha*}\right\}_{\substack{\beta,\delta<\alpha+1 \\ \beta\neq\delta}},\left\{w_{\beta}^{\alpha*}\right\}_{\substack{\beta\in\mathcal{I}_{2} \\ \beta<\alpha+1}}\right)\subset Ker(i_{\alpha+1}^{*}).
\end{equation*}
Now, we will prove that $Ker(i_{\alpha+1}^{*})\subset \mathcal{I}_{\alpha+1}$. Note that $i_{\alpha+1}^{*}: A^{\bullet}(Z_{\alpha})\rightarrow A^{\bullet}(C_{\alpha+1})$ is homogenous, so $ker(i_{\alpha+1}^{*})$ is an homogenous ideal, and $\mathcal{I}_{\alpha+1}$ is an homogenous ideal too by construction. Let us suppose that 
\begin{equation*}
Q\left[h^{\alpha*},e_{1}^{\alpha*},...,e_{\alpha}^{\alpha*},w_{1}^{\alpha*},...,w_{\alpha}^{\alpha*}\right]\in Ker(i_{\alpha+1}^{*})/\mathcal{I}_{\alpha+1},
\end{equation*}
with $deg(Q)=\eta$. Then $\eta\leq 1$, since all polynomials of weighted degree $2$ are all in $\mathcal{I}_{\alpha+1}$, and $Q\left[h^{\alpha*},e_{1}^{\alpha*},...,e_{\alpha}^{\alpha*},w_{1}^{\alpha*},...,w_{\alpha}^{\alpha*}\right]$ must be of the form $Q\left[h^{\alpha*},e_{1}^{\alpha*},...,e_{\alpha}^{\alpha*},w_{1}^{\alpha*},...,w_{\alpha}^{\alpha*}\right]=b_{0}h^{\alpha*}+\sum_{i=1}^{\alpha}b_{i}e_{i}^{\alpha*}mod(\mathcal{I}_{\alpha+1})$. But then $b_{i}=0$ for $i=0,...,\alpha$ since $\left\{a_{0}\mu_{0}+\sum_{\beta=1}^{\alpha}a_{\beta}\mu_{\beta}\right\}$ is the minimum set of relations of the finitely generated free abelian group $\mathcal{S}_{\alpha+1}$. \\
Before going on, we should distinguish between two possible cases, that is:
\begin{enumerate}
\item either $i_{\alpha+1}^{*}$ is surjective, \label{CaseSur}
\item or $i_{\alpha+1}^{*}$ is not surjective. \label{CaseNSur}
\end{enumerate}
In Case \ref{CaseSur} we have that $A^{\bullet}(C_{\alpha+1})\cong A^{\bullet}(Z_{\alpha})/ker i_{\alpha+1}^{*}$. Moreover, there must exist a relation of the form
\begin{equation*}
a_{0}i_{\alpha+1}^{*}h^{\alpha*}+\sum a_{\beta}i_{\alpha+1}^{*}e_{\beta}^{\alpha*}=\left[P\right],
\end{equation*}
where $\left[P\right]\in A^{1}(C_{\alpha+1})$ is a generator of $A^{1}(C_{\alpha+1})$, so in this case we can conclude that:
\begin{equation*}
A^{\bullet}(C_{\alpha+1})\cong A^{\bullet}(Z_{\alpha})\left[P\right]/(ker i_{\alpha+1}^{*}, a_{0}h^{\alpha*}+\sum a_{\beta}e_{\beta}^{\alpha*}-[P]).
\end{equation*}
However, Case \ref{CaseNSur} is a bit more tricky. In particular, we have that $A^{\bullet}(C_{\alpha+1})$ is isomorphic to
\begin{equation*}
A^{\bullet}(C_{\alpha+1})\cong A^{\bullet}(Z_{\alpha})\left[P\right]/(ker i_{\alpha+1}^{*},h^{\alpha*}-\mu_{0}\left[P\right],\left\{e_{\beta}^{\alpha*}-\mu_{\beta}\left[P\right]\right\}_{\beta=1}^{\alpha}, (\left[P\right])^{2}).
\end{equation*}
Now, since $E_{\alpha+1}^{\alpha+1}$ is isomorphic to the projective bundle $P(N_{C_{\alpha+1}/Z_{\alpha}})$ over $C_{\alpha+1}$, then by Theorem \ref{ThmCRPB} it follows that:
\begin{equation*}
A^{\bullet}(E_{\alpha+1}^{\alpha+1})\cong A^{\bullet}(C_{\alpha+1})\left[\varsigma_{\alpha+1}\right]/(\varsigma_{\alpha+1}^{2}+c_{1}(N_{C_{\alpha+1}/Z_{\alpha}})\varsigma_{\alpha+1}\cdot p).
\end{equation*}

It can be proved that in both Cases \ref{CaseSur} and \ref{CaseNSur} it is satisfied the following inclusion $\mathcal{J}_{\alpha+1}\subset Ker\: f_{\alpha+1}$. Recall that 
\begin{multline*}
\mathcal{J}_{\alpha+1}=\left(ker\: i_{\alpha+1}^{*}\cdot e_{\alpha+1}^{\alpha+1},h^{\alpha+1*}\cdot e_{\alpha+1}^{\alpha+1}-\mu_{0}w_{\alpha+1}^{\alpha+1},\left\{e_{\beta}^{\alpha*}\cdot e_{\alpha+1}^{\alpha+1}-\mu_{\beta}w_{\alpha+1}^{\alpha+1}\right\},(w_{\alpha+1}^{\alpha+1})^{2},\right. \\
  h^{\alpha+1*}\cdot w_{\alpha+1}^{\alpha+1},\left\{e_{\beta}^{\alpha*}\cdot w_{\alpha+1}^{\alpha+1}\right\}_{\beta=1}^{\alpha}, (e_{\alpha+1}^{\alpha+1})^{2}-c_{1}(N_{C_{\alpha+1}/Z_{\alpha}})w_{\alpha+1}^{\alpha+1}+\left[C_{\alpha+1}\right], \\
\left. e_{\alpha+1}^{\alpha+1}\cdot w_{\alpha+1}^{\alpha+1}+(h^{\alpha+1*})^{3}\right)
\end{multline*}
Firstly, since $f_{\alpha+1}$ is a ring homomorphism then we have that $f_{\alpha+1}(x\cdot y)=f_{\alpha+1}(x)\cdot f_{\alpha+1}(y)$, so the inclusion
\begin{multline*}
\left(ker\: i_{\alpha+1}^{*}\cdot e_{\alpha+1}^{\alpha+1},h^{\alpha+1*}\cdot e_{\alpha+1}^{\alpha+1}-\mu_{0}w_{\alpha+1}^{\alpha+1},\left\{e_{\beta}^{\alpha*}\cdot e_{\alpha+1}^{\alpha+1}-\mu_{\beta}w_{\alpha+1}^{\alpha+1}\right\},(w_{\alpha+1}^{\alpha+1})^{2}, \right. \\
 \left. h^{\alpha+1*}\cdot w_{\alpha+1}^{\alpha+1},\left\{e_{\beta}^{\alpha*}\cdot w_{\alpha+1}^{\alpha+1}\right\}_{\beta=1}^{\alpha}\right)\subset ker f_{\alpha+1},
\end{multline*}
follows directly from Proposition \ref{ProMRBU}. Moreover, the inclusion 
\begin{equation*}
\left((e_{\alpha+1}^{\alpha+1})^{2}-c_{1}(N_{C_{\alpha+1}/Z_{\alpha}})w_{\alpha+1}^{\alpha+1}+\left[C_{\alpha+1}\right]\right)\subset ker f_{\alpha+1},
\end{equation*}
is a direct consequence of the key formula (see Proposition \ref{ProGenCRBU}). Finally, the inclusion $(e_{\alpha+1}^{\alpha+1}\cdot w_{\alpha+1}^{\alpha+1}+(h^{\alpha+1*})^{3})\subset ker f_{\alpha+1}$ follows from the key formula and the birational invariance of $A_{0}(Z_{i})$ (see \cite[Example 16.1.11]{Fulton98}).

In order to continue with the proof, let us recall that by Proposition \ref{ProGenCRBU} we have the following exact sequence:
\begin{equation*}
0\xrightarrow{} A^{\bullet}(C_{\alpha+1})\xrightarrow{l} A^{\bullet}(E_{\alpha+1}^{\alpha+1})\oplus A^{\bullet}(Z_{\alpha})\xrightarrow{m} A^{\bullet}(Z_{\alpha+1})\xrightarrow{} 0
\end{equation*}
where $l(x)=((g_{\alpha+1}^{*}c_{1}(N_{C_{\alpha+1}/Z_{\alpha}})+\varsigma_{\alpha+1})\cap g_{\alpha+1}^{*}(x), i_{\alpha+1 *}(x))$ and $m(y)=(-j_{\alpha+1 *}(y), \pi_{\alpha+1}^{*}(y))$.
Let us now define
\begin{equation*}
R_{\alpha+1}:=A^{\bullet}(Z_{\alpha})\left[e_{\alpha+1}^{\alpha+1},w_{\alpha+1}^{\alpha+1}\right]/\mathcal{J}_{\alpha+1},
\end{equation*}
and a group homorphism $\gamma: A^{\bullet}(E_{\alpha+1}^{\alpha+1})\oplus A(Z_{\alpha})\rightarrow R_{\alpha+1}$ such that $\gamma(x,y)=-h_{\alpha+1}(x)+\pi_{\alpha+1}^{*}(y)$, where
$$
h_{\alpha+1}(x)=\begin{cases}
(e_{\alpha+1}^{\alpha+1})^{\eta+1} & \text{if}\enspace x=(-1)^{\eta}(\varsigma_{\alpha+1})^{\eta}\enspace\text{for}\enspace \eta\geq 1, \\
w_{\alpha+1}^{\alpha+1} & \text{if}\enspace x=p, \\
(d_{0}h^{\alpha*}+\sum_{\beta=1}^{\alpha}d_{\beta}e_{\beta}^{\alpha*})^{\lambda-1}\cdot w_{\alpha+1}^{\alpha+1} & \text{if}\enspace x=(p)^{\lambda}\enspace\text{for}\enspace \lambda\geq 2, \\
h_{\alpha+1}((p)^{\lambda})\cdot(e_{\alpha+1}^{\alpha+1})^{\eta}  & \text{if}\enspace x=(p)^{\lambda}\cdot(\varsigma_{\alpha+1})^{\eta}\enspace\text{for}\enspace \lambda\geq 2, \eta\geq 1,                                                    
\end{cases} $$,
giving factorization of $m: A^{\bullet}(E_{\alpha+1}^{\alpha+1})\oplus A^{\bullet}(Z_{\alpha})\rightarrow A^{\bullet}(Z_{\alpha+1})$, that is,
\begin{equation*}
\xymatrix{ A^{\bullet}(E_{\alpha+1}^{\alpha+1})\oplus A^{\bullet}(Z_{\alpha})\ar[r]^{m}\ar[d]_{\gamma} & A^{\bullet}(Z_{\alpha+1})\ar[d] \\
R_{\alpha+1}\ar[r]^(.3){\varphi_{\alpha+1}} & A^{\bullet}(Z_{\alpha})\left[e_{\alpha+1}^{\alpha+1},w_{\alpha+1}^{\alpha+1}\right]/ker f_{\alpha+1}\ar[u]}
\end{equation*}

In order to prove that $R_{\alpha+1}\cong A^{\bullet}(Z_{\alpha+1})$, that is $\varphi_{\alpha+1}$ is an isomorphism, it suffices to verify that $\gamma\circ l=0$. Choose $\left[C_{\alpha+1}\right]\in A^{0}(C_{\alpha+1})$. Then $l(\left[C_{\alpha+1}\right])=(\varsigma_{\alpha}+c_{1}(N_{C_{\alpha}/Z_{\alpha-1}})p, \left[C_{\alpha+1}\right])$, and
\begin{equation*} 
\gamma(l(\left[C_{\alpha+1}\right]))=(e_{\alpha+1}^{\alpha+1})^{2}-c_{1}(N_{C_{\alpha+1}/Z_{\alpha}})w_{\alpha}+\left[C_{\alpha+1}\right]=0.
\end{equation*}
Choose now $\left[P_{\alpha+1}\right]\in A^{1}(C_{\alpha+1})$. Then $l(\left[P_{\alpha}\right])=(\varsigma_{\alpha+1}\cdot p, (h^{\alpha*})^{3})$ and
\begin{equation*}
\gamma(l(\left[P_{\alpha+1}\right]))=e_{\alpha+1}^{\alpha+1}\cdot w_{\alpha+1}^{\alpha+1}+(h^{\alpha*})^{3}=0.
\end{equation*}
\end{proof}

\begin{cor}\label{CorChowRingSeqRCPBU}
The Chow ring of the sky $A^{\bullet}(Z_{s})$ is isomorphic to
\begin{equation*}
A^{\bullet}(Z_{s})\cong\frac{\mathbb{Z}\left[h^{s*}, \left\{e_{\alpha}^{s*}\right\}_{\alpha\in\mathcal{I}_{1}}, \left\{e_{\beta}^{s*}, w_{\beta}^{s*}\right\}_{\beta\in\mathcal{I}_{2}}\right]}{\mathcal{A}},
\end{equation*}
where
\begin{multline*}
\mathcal{A}=\left((h^{s*})^{4},\left\{\left\{h^{s*}\cdot e_{\alpha}^{s*}\right\}, \left\{e_{\alpha}^{s*}\cdot e_{\beta}^{s*}\right\}_{\alpha\neq\beta},\left\{-(e_{\alpha}^{s*})^{3}+(h^{s*})^{n}\right\}\right\}_{\alpha,\beta\in\mathcal{I}_{1}}, \right. \\
\bigl\{ ker i_{\alpha}^{s*}\cdot e_{\alpha}^{s*},h^{s*}\cdot e_{\alpha}^{s*}-\mu_{0}w_{\alpha}^{s*},\left\{e_{\beta}^{s*}\cdot e_{\alpha}^{s*}-\mu_{\beta}w_{\alpha}^{s*}\right\}_{\beta<\alpha},(w_{\alpha}^{s*})^{2}, h^{s*}\cdot w_{\alpha}^{s*}, \\
\left\{e_{\beta}^{s*}\cdot w_{\alpha}^{s*}\right\}_{\beta<\alpha},\left. (e_{\alpha}^{s*})^{2}-c_{1}(N_{C_{\alpha}/Z_{\alpha-1}})w_{\alpha}^{s*}+\left[C_{\alpha}\right]^{s*}, e_{\alpha}^{s*}\cdot w_{\alpha}^{s*}+(h^{s*})^{3}\bigr\}_{\alpha,\beta\in\mathcal{I}_{2}}\right).
\end{multline*}
\end{cor}

\begin{proof}
By \cite[Theorem 3.3.]{Camazon24} and Theorem \ref{ThmCRalpha+1} we know that
\begin{enumerate}
\item if $dim(C_{\alpha+1})=0$, that is $C_{\alpha+1}=P_{\alpha+1}$, then
\begin{equation*}
A^{\bullet}(Z_{\alpha+1})\cong\frac{A^{\bullet}(Z_{\alpha})\left[e_{\alpha+1}^{\alpha+1*}\right]}{(h^{\alpha+1*}\cdot e_{\alpha+1}^{\alpha+1*},\left\{e_{i}^{\alpha+1*}\cdot e_{\alpha+1}^{\alpha+1*}\right\}_{i=1}^{\alpha}, -(e_{\alpha+1}^{\alpha+1*})^{3}+(h^{\alpha+1*})^{3})},
\end{equation*}
\item and if $dim(C_{\alpha+1})=1$, that is $C_{\alpha+1}=\mathcal{C}_{\alpha+1}$, then
\begin{equation*}
A^{\bullet}(Z_{\alpha+1})\cong\frac{A^{\bullet}(Z_{\alpha})\left[e_{\alpha+1}^{\alpha+1*},w_{\alpha+1}^{\alpha+1*}\right]}{\mathcal{J}_{\alpha+1}}.
\end{equation*}
\end{enumerate}
So, since $A^{\bullet}(Z_{0})\cong\frac{\mathbb{Z}\left[h\right]}{(h)^{4}}$ by Theorem \ref{ThmPn}, the result follows directly by induction.
\end{proof}

We would like to point out that there exists substantial differences when considering sequences of blow-ups at higher dimensional centers instead of just point blow-ups.
\begin{thm}\cite[Theorem 4.]{EisenbudVandeVen81}
Given any integer $\gamma\geq 4$, there exist smooth rational curves $C$ of degree $\gamma$ in $\mathbb{P}^{3}$ with normal bundle isomorphic to $\mathcal{O}_{C}(2\gamma-1-a)\oplus\mathcal{O}_{C}(2\gamma-1+a)$ if and only if $\left|a\right|\leq \gamma-4$.
\end{thm} 
As a result, given a rational curve $\mathcal{C}$ of a given degree $\gamma\geq 4$, then two different embeddings $\mathcal{C}\hookrightarrow\mathbb{P}^{3}$ and the subsequent blow-ups give rise to varieties with non-isomorphic Chow rings. Moreover, the following result make a huge difference with respect to the study of sequences of point blow-ups. We recall that in \cite{Camazon24} it was proved that the skies of two sequences of point blow-ups of the same length have isomorphic Chow rings.

\begin{cor}
Given two sequences of blow-ups $(Z_{s},...,Z_{0},\pi)$, $(Z_{s}^{'},...,Z_{0}^{'},\pi^{'})$, where $Z_{0}\cong Z_{0}^{'}$, with the same length, $s=s^{'}$, and proximity relations, then $A^{\bullet}(Z_{s})$ and $A^{\bullet}(Z_{s}^{'})$ may be non-isomorphic.
\end{cor}

Now, a natural question arises: is there any example of two sequences of point and rational curve blow-ups such that the Chow rings of their skies are isomorphic? Within the next examples we will explore some necessary conditions for the existence of such an isomorphism. For the sake simplicity, and in order clarify the exposition, we will just consider sequences of two blow-ups:
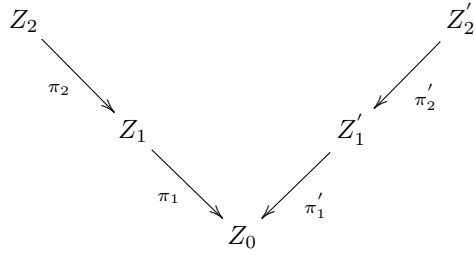
\begin{equation*}
\xymatrix{Z_{2}\ar[rd]_{\pi_{2}} & & & & Z_{2}^{'}\ar[ld]^{\pi_{2}^{'}} \\
 & Z_{1}\ar[rd]_{\pi_{1}} & & Z_{1}^{'}\ar[ld]^{\pi_{1}^{'}} & \\
 & & Z_{0} & &}
\end{equation*}
\begin{ex}\label{Ex1}
Let $\pi_{1}: Z_{1}\rightarrow Z_{0}$ be the blow-up with center $C_{1}$ a rational curve of degree $\gamma_{1}$, with $\gamma_{1}\geq 4$, and $\pi_{2}: Z_{2}\rightarrow Z_{1}$ be the blow-up with center $C_{2}$ the section corresponding to the line subbundle $\mathcal{O}_{C_{1}}(2\gamma_{1}-1-a-n)$, that is, $C_{2}\xrightarrow{} C_{1}$, where $n=0$ or $n\geq 2a$. Then, we know that $A^{\bullet}(Z_{2})$ is generated by
\begin{equation*}
\begin{cases}
\left\{h^{2*},e_{1}^{2*},e_{2}^{2*}\right\} & \text{in codimension}\enspace 1, \\
\left\{(h^{2*})^{2},(w_{1}^{2*})^{2},w_{2}^{2*}\right\} & \text{in codimension}\enspace 2, \\
\left\{(h^{2*})^{3}\right\} & \text{in codimension}\enspace 3,
\end{cases}
\end{equation*}
and it follows from Corollary \ref{CorChowRingSeqRCPBU} that the Chow ring $A^{\bullet}(Z_{2})$ is isomorphic to
\begin{equation*}
A^{\bullet}(Z_{2})\cong\frac{\mathbb{Z}\left[h^{2*},e_{1}^{2*},w_{1}^{2*},e_{2}^{2*},w_{2}^{2*}\right]}{\mathcal{A}},
\end{equation*}
where
\begin{multline*}
\mathcal{A}=\left((h^{2*})^{4}, (h^{2*})^{2}\cdot e_{1}^{2*}, h^{2*}\cdot e_{1}^{2*}-\gamma_{1}w_{1}^{2*}, h^{2*}\cdot w_{1}^{2*}, (w_{1}^{2*})^{2}, \right. \\
(e_{1}^{2*})^{2}-(4\gamma_{1}-2)w_{1}^{2*}+\gamma_{1}(h^{2*})^{2}, e_{1}^{2*}\cdot w_{1}^{2*}+(h^{2*})^{3}, (h^{2*})^{2}\cdot e_{2}^{2*}, (e_{1}^{2*})^{2}\cdot e_{2}^{2*}, \\
h^{2*}\cdot e_{1}^{2*}\cdot e_{2}^{2*}, w_{1}^{2*}\cdot e_{2}^{2*}, h^{2*}\cdot e_{2}^{2*}-\gamma_{1}w_{2}^{2*}, e_{1}^{2*}\cdot e_{2}^{2*}-(2\gamma_{1}-1-a-n)w_{2}^{2*}, h^{2*}\cdot w_{2}^{2*}, e_{1}^{2*}\cdot w_{2}^{2}, \\
w_{1}^{2*}\cdot w_{2}^{2*}, (w_{2}^{2*})^{2}, (e_{2}^{2*})^{2}-(2\gamma_{1}-1+a+n)w_{2}^{2*}-(e_{1}^{2*})^{2}+(2\gamma_{1}-1+a+n)w_{1}^{2*}, \\
\left. e_{2}^{2*}\cdot w_{2}^{2*}+(h^{2*})^{3}\right).
\end{multline*}
Let $\pi_{1}^{'}: Z_{1}^{'}\rightarrow Z_{0}$ be the blow-up with center $C_{1}^{'}$ a rational curve of degree $\gamma_{1}^{'}$, with $\gamma_{1}\geq 4$, and $\pi_{2}^{'}: Z_{2}^{'}\rightarrow Z_{1}$ be the blow-up with center $C_{2}^{'}$ the section corresponding to the line subbundle $\mathcal{O}_{C_{1}}(2\gamma_{1}^{'}-1-a^{'}-n^{'})$, that is, $C_{2}^{'}\xrightarrow{} C_{1}^{'}$, where where $n^{'}=0$ or $n^{'}\geq 2a$. Then, we know that $A^{\bullet}(Z_{2}^{'})$ is generated by
\begin{equation*}
\begin{cases}
\left\{h^{2'*},e_{1}^{2'*},e_{2}^{2'*}\right\} & \text{in codimension}\enspace 1, \\
\left\{(h^{2'*})^{2},(w_{1}^{2'*})^{2},w_{2}^{2'*}\right\} & \text{in codimension}\enspace 2, \\
\left\{(h^{2'*})^{3}\right\} & \text{in codimension}\enspace 3,
\end{cases}
\end{equation*}
and it follows from Corollary \ref{CorChowRingSeqRCPBU} that the Chow ring $A^{\bullet}(Z_{2}^{'})$ is isomorphic to:
\begin{equation*}
A^{\bullet}(Z_{2}^{'})\cong\frac{\mathbb{Z}\left[h^{'2*},e_{1}^{2'*},w_{1}^{2'*},e_{2}^{2'*},w_{2}^{2'*}\right]}{\mathcal{A}^{'}},
\end{equation*}
where
\begin{multline*}
\mathcal{A}^{'}=\left((h^{2'*})^{4}, (h^{'2*})^{2}\cdot e_{1}^{2'*}, h^{2'*}\cdot e_{1'}^{2*}-\gamma_{1}^{'}w_{1}^{2'*}, h^{2'*}\cdot w_{1}^{2'*}, (w_{1}^{2'*})^{2}, \right. \\
(e_{1}^{2'*})^{2}-(4\gamma_{1}^{'}-2)w_{1}^{2'*}+\gamma_{1}^{'}(h^{2'*})^{2}, e_{1}^{2'*}\cdot w_{1}^{2'*}+(h^{2'*})^{3}, (h^{2'*})^{2}\cdot e_{2}^{2'*}, (e_{1}^{2'*})^{2}\cdot e_{2}^{2'*}, \\
 h^{2'*}\cdot e_{1}^{2'*}\cdot e_{2}^{2'*}, w_{1}^{2'*}\cdot e_{2}^{2'*}, h^{2'*}\cdot e_{2}^{2'*}-\gamma_{1}^{'}w_{2}^{2'*}, e_{1}^{2*}\cdot e_{2}^{2*}-(2\gamma_{1}^{'}-1-a^{'}-n^{'})w_{2}^{2'*}, h^{2'*}\cdot w_{2}^{2'*}, \\
e_{1}^{2'*}\cdot w_{2}^{2'*}, w_{1}^{2'*}\cdot w_{2}^{2'*}, (w_{2}^{2'*})^{2}, (e_{2}^{2'*})^{2}-(2\gamma_{1}^{'}-1+a^{'}+n^{'})w_{2}^{2'*}-(e_{1}^{2'*})^{2}+(2\gamma_{1}^{'}-1+a^{'}+n^{'})w_{1}^{2'*}, \\
\left. e_{2}^{2'*}\cdot w_{2}^{2'*}+(h^{2'*})^{3}\right).
\end{multline*}
Let us define a graded ring homomorphism $\phi: A^{\bullet}(Z_{2})\rightarrow A^{\bullet}(Z_{2}^{'})$, so
\begin{align*}
& \phi(h^{2*})=a_{0}h^{2'*}+a_{1}e_{1}^{2'*}+a_{2}e_{2}^{2'*}, \\ 
& \phi(e_{1}^{2*})=b_{0}h^{2'*}+b_{1}e_{1}^{2'*}+b_{2}e_{2}^{2'*}, \\
& \phi(e_{2}^{2*})=c_{0}h^{2'*}+c_{1}e_{1}^{2'*}+c_{2}e_{2}^{2'*}, \\
& \phi(w_{1}^{2*})=d_{0}(h^{2'*})^{2}+d_{1}w_{1}^{2'*}+d_{2}w_{2}^{2'*}, \\
& \phi(w_{2}^{2*})=f_{0}(h^{2'*})^{2}+f_{1}w_{1}^{2'*}+f_{2}w_{2}^{2'*}.
\end{align*}
Since $A^{\bullet}(Z_{2}^{'})$ is generated by $\left\{(h^{2'*})^{3}\right\}$ in codimension $3$, then we have that
\begin{equation*}
\phi\left((h^{2*})^{3}\right)=(\phi(h^{2*}))^{3}=(h^{2'*})^{3},
\end{equation*}
so we can conclude that $\phi(h^{2*})=h^{2'*}$, that is, $a_{0}=1, a_{1}=a_{2}=0$. If $\phi$ would be a graded isomorphism, then by \cite[9 Theorem.]{Eick16} the following relations will hold:
\begin{align}
& \phi(h^{2*}\cdot  e_{1}^{2*}-\gamma_{1}w_{1}^{2*})=\phi(h^{2*})\cdot\phi(e_{1}^{2*})-\gamma_{1}\phi(w_{1}^{2*})=0, \label{EqRel1.2} \\
& \phi(h^{2*}\cdot  e_{2}^{2*}-\gamma_{1}w_{2}^{2*})=\phi(h^{2*})\cdot\phi(e_{2}^{2*})-\gamma_{1}\phi(w_{2}^{2*})=0, \label{EqRel2.2} \\
& \phi(h^{2*}\cdot  w_{1}^{2*})=\phi(h^{2*})\cdot\phi(w_{1}^{2*})=0, \label{EqRel3.2} \\
& \phi(h^{2*}\cdot  w_{2}^{2*})=\phi(h^{2*})\cdot\phi(w_{2}^{2*})=0, \label{EqRel4.2} \\
& \phi(e_{1}^{2*}\cdot  e_{2}^{2*}-(2\gamma_{1}-1-a-n)w_{2}^{2*})=\phi(e_{1}^{2*})\cdot\phi(e_{2}^{2*})-(2\gamma_{1}-1-a-n)\phi(w_{2}^{2*})=0, \label{EqRel5.2} \\
& \phi(e_{1}^{2*}\cdot w_{2}^{2*})=\phi(e_{1}^{2*})\cdot\phi(w_{2}^{2*})=0, \label{EqRel6.2} \\
& \phi(w_{1}^{2*}\cdot e_{2}^{2*})=\phi(w_{1}^{2*})\cdot\phi(e_{2}^{2*})=0, \label{EqRel7.2} \\
& \phi(w_{1}^{2*}\cdot w_{2}^{2*})=\phi(w_{1}^{2*})\cdot\phi(w_{2}^{2*})=0, \label{EqRel8.2} \\
& \phi(e_{1}^{2*}\cdot w_{1}^{2*}+(h^{2*})^{3})=\phi(e_{1}^{2*})\cdot\phi(w_{1}^{2*})+\phi(h^{2*})^{3}=0, \label{EqRel9.2} \\
& \phi(e_{2}^{2*}\cdot w_{2}^{2*}+(h^{2*})^{3})=\phi(e_{2}^{2*})\cdot\phi(w_{2}^{2*})+\phi(h^{2*})^{3}=0, \label{EqRel10.2} \\
& \phi\left((e_{1}^{2*})^{2}-(4\gamma_{1}-2)w_{1}^{2*}+\gamma_{1}(h^{2*})^{2}\right)=\phi(e_{1}^{2*})^{2}-(4\gamma_{1}-2)\phi(w_{1}^{2*})+\gamma_{1}\phi(h^{2*})^{2}=0. \label{EqRel11.2} \\
& \phi\left((e_{2}^{2*})^{2}-(2\gamma_{1}-1+a+n)w_{2}^{2*}-(e_{1}^{2*})^{2}+(2\gamma_{1}-1+a+n)w_{1}^{2*})\right)=\nonumber \\ 
& \phi(e_{2}^{2*})^{2}-(2\gamma_{1}-1-a-n)\phi(w_{2}^{2*})-\phi(e_{1}^{2*})^{2}+(2\gamma_{1}-1+a+n)\phi(w_{1}^{2*})=0. \label{EqRel12.2}
\end{align}
From Equations (\ref{EqRel1.2}), (\ref{EqRel2.2}), (\ref{EqRel3.2}) and (\ref{EqRel4.2}), we have that: 
\begin{align*}
b_{0} & =d_{0}=0, \\
d_{1} & =\frac{\gamma_{1}^{'}}{\gamma_{1}}b_{1}, \\
d_{2} & =\frac{\gamma_{1}^{'}}{\gamma_{1}}b_{2}, \\
c_{0} & =f_{0}=0, \\
f_{1} & =\frac{\gamma_{1}^{'}}{\gamma_{1}}c_{1}, \\
f_{2} & =\frac{\gamma_{1}^{'}}{\gamma_{1}}c_{2}.
\end{align*}
Moreover, Equations (\ref{EqRel6.2}) and (\ref{EqRel7.2}) implies that $b_{1}f_{1}=-b_{2}f_{2}$. Now, as a consequence of Equations (\ref{EqRel9.2}) and (\ref{EqRel10.2}) the following relations hold $b_{1}d_{1}+b_{2}d_{2}=1$ and $c_{1}f_{1}+c_{2}f_{2}=1$. Finally, it follows from Equations (\ref{EqRel5.2}) and (\ref{EqRel11.2}) that:
\begin{align*}
& b_{1}^{2}(4\gamma_{1}^{'}-2)+b_{2}^{2}(2\gamma_{1}^{'}-1-a^{'}-n^{'})-d_{1}(4\gamma_{1}-2)=0, \\ 
& b_{1}b_{2}(2\gamma_{1}^{'}-1-a^{'}-n^{'})+b_{2}^{2}(2\gamma_{1}^{'}-1+a^{'}+n^{'})-d_{2}(4\gamma_{1}-2)=0, \\
& b_{1}c_{1}(4\gamma_{1}^{'}-2)+b_{2}c_{2}(2\gamma_{1}^{'}-1-a^{'}-n^{'})-f_{1}(2\gamma_{1}-1-a-n)=0, 
\end{align*} 
\begin{multline*}
\quad\; b_{1}c_{2}(2\gamma_{1}^{'}-1-a^{'}-n^{'})+b_{2}c_{1}(2\gamma_{1}^{'}-1-a^{'}-n^{'})+b_{2}c_{2}(2\gamma_{1}^{'}-1+a^{'}+n^{'})- \\
f_{2}(2\gamma_{1}-1-a-n)=0. 
\end{multline*}
We can conclude that if $\phi$ would be a graded isomorphism, then $\gamma_{1}=\gamma_{1}^{'}$, $a+n=a^{'}+n^{'}$, $b_{1}=d_{1}=c_{2}=f_{2}=1$ and $b_{2}=d_{2}=c_{1}=f_{1}=0$.
\end{ex}

\begin{ex}\label{Ex2}
Let $\mathcal{C}, \mathcal{D}\subset\mathbb{P}^{3}$ be two smooth rational curve of degree $\gamma_{1}$ and $\gamma_{2}$, respectively, and $\pi_{1}: Z_{1}\rightarrow Z_{0}=\mathbb{P}^{3}$ the blow-up of $\mathbb{P}^{3}$ with center $C_{1}=\mathcal{C}$. If we denote by $\widetilde{\mathcal{D}}$ to the strict transform of $\mathcal{D}$ $in Z_{1}$, let $\pi_{2}: Z_{2}\rightarrow Z_{1}$ be the blow-up of $Z_{1}$ with center $C_{2}=\widetilde{\mathcal{D}}$, that is, $C_{2}\xrightarrow{t} C_{1}$. Let $\beta_{1,2}$ denote the cardinality of the set $\mathcal{C}\cap\mathcal{D}$. Then, we know that $A^{\bullet}(Z_{2})$ is generated by
\begin{equation*}
\begin{cases}
\left\{h^{2*},e_{1}^{2*},e_{2}^{2*}\right\} & \text{in codimension}\enspace 1, \\
\left\{(h^{2*})^{2},(w_{1}^{2*})^{2},w_{2}^{2*}\right\} & \text{in codimension}\enspace 2, \\
\left\{(h^{2*})^{3}\right\} & \text{in codimension}\enspace 3,
\end{cases}
\end{equation*}
and it follows from Corollary \ref{CorChowRingSeqRCPBU} that the Chow ring $A^{\bullet}(Z_{2})$ is isomorphic to:
\begin{equation*}
A^{\bullet}(Z_{2})\cong\frac{\mathbb{Z}\left[h^{2*},e_{1}^{2*},w_{1}^{2*},e_{2}^{2*},w_{2}^{2*}\right]}{\mathcal{A}},
\end{equation*}
where
\begin{multline*}
\mathcal{A}=\left((h^{2*})^{4}, (h^{2*})^{2}\cdot e_{1}^{2*}, h^{2*}\cdot e_{1}^{2*}-\gamma_{1} w_{1}^{2*}, h^{2*}\cdot w_{1}^{2*}, (w_{1}^{2*})^{2}, \right. \\
(e_{1}^{2*})^{2}-(4\gamma_{1}-2)w_{1}^{2*}+\gamma_{1}(h^{2*})^{2}, e_{1}^{2*}\cdot w_{1}^{2*}+(h^{2*})^{3}, (h^{2*})^{2}\cdot e_{2}^{2*}, (e_{1}^{2*})^{2}\cdot e_{2}^{2*}, \\
h^{2*}\cdot e_{1}^{2*}\cdot e_{2}^{2*}, w_{1}^{2*}\cdot e_{2}^{2*}, h^{2*}\cdot e_{2}^{2*}-\gamma_{2} w_{2}^{2*}, e_{1}^{2*}\cdot e_{2}^{2*}-\beta_{1,2}w_{2}^{2*}, h^{2*}\cdot w_{2}^{2*}, \\
e_{1}^{2*}\cdot w_{2}^{2}, w_{1}^{2*}\cdot w_{2}^{2*}, (w_{2}^{2*})^{2},(e_{2}^{2*})^{2}-(4\gamma_{2}-2-\beta_{1,2})w_{2}^{2*}+\eta(h^{2*})^{2}-\beta_{1,2}w_{1}^{2*}, \\
\left. e_{2}^{2*}\cdot w_{2}^{2*}+(h^{2*})^{3}\right).
\end{multline*}
Let $\mathcal{C}^{'}, \mathcal{D}^{'}\subset\mathbb{P}^{3}$ be two smooth rational curve of degree $\gamma_{1}^{'}$ and $\gamma_{2}^{'}$, respectively, and $\pi_{1}^{'}: Z_{1}^{'}\rightarrow Z_{0}$ be the blow-up of $\mathbb{P}^{3}$ with center $C_{1}^{'}=\mathcal{C}^{'}$, $\widetilde{\mathcal{D}}^{'}$ to the strict transform of $\mathcal{D}^{'}$ $in Z_{1}^{'}$, and $\pi_{2}: Z_{2}^{'}\rightarrow Z_{1}^{'}$ be the blow-up of $Z_{1}^{'}$ with center $C_{2}^{'}=\widetilde{\mathcal{D}}^{'}$, that is $C_{2}^{'}\xrightarrow C_{1}^{'}$. Let $\beta_{1,2}^{'}$ denote the cardinality of the set $\mathcal{C}^{'}\cap\mathcal{D}^{'}$. Then, we know that $A^{\bullet}(Z_{2}^{'})$ is generated by
\begin{equation*}
\begin{cases}
\left\{h^{2'*},e_{1}^{2'*},e_{2}^{2'*}\right\} & \text{in codimension}\enspace 1, \\
\left\{(h^{2'*})^{2},(w_{1}^{2'*})^{2},w_{2}^{2'*}\right\} & \text{in codimension}\enspace 2, \\
\left\{(h^{2'*})^{3}\right\} & \text{in codimension}\enspace 3,
\end{cases}
\end{equation*}
and it follows from Corollary \ref{CorChowRingSeqRCPBU} that the Chow ring $A^{\bullet}(Z_{2}^{'})$ is isomorphic to:
\begin{equation*}
A^{\bullet}(Z_{2}^{'})\cong\frac{\mathbb{Z}\left[h^{'2*},e_{1}^{2'*},w_{1}^{2'*},e_{2}^{2'*},w_{2}^{2'*}\right]}{\mathcal{A}^{'}},
\end{equation*}
where
\begin{multline*}
\mathcal{A}^{'}=\left((h^{2'*})^{4}, (h^{'2*})^{2}\cdot e_{1}^{2'*}, h^{2'*}\cdot e_{1'}^{2*}-\gamma_{1}^{'} w_{1}^{2'*}, h^{2'*}\cdot w_{1}^{2'*}, (w_{1}^{2'*})^{2}, \right. \\
(e_{1}^{2'*})^{2}-(4\gamma_{1}^{'}-2)w_{1}^{2'*}+\eta(h^{2'*})^{2}, e_{1}^{2'*}\cdot w_{1}^{2'*}+(h^{2'*})^{3}, (h^{2'*})^{2}\cdot e_{2}^{2'*}, (e_{1}^{2'*})^{2}\cdot e_{2}^{2'*}, \\
h^{2'*}\cdot e_{1}^{2'*}\cdot e_{2}^{2'*}, w_{1}^{2'*}\cdot e_{2}^{2'*}, h^{2'*}\cdot e_{2}^{2'*}-\gamma_{2}^{'} w_{2}^{2'*}, e_{1}^{2*}\cdot e_{2}^{2*}-\beta_{1,2}^{'}w_{2}^{2'*}, \\ 
h^{2'*}\cdot w_{2}^{2'*}, e_{1}^{2'*}\cdot w_{2}^{2'*}, w_{1}^{2'*}\cdot w_{2}^{2'*}, (w_{2}^{2'*})^{2}, (e_{2}^{2'*})^{2}-(4\gamma_{2}^{'}-2-\beta_{1,2})w_{2}^{2'*}+\gamma(h^{2'*})^{2}-\beta_{1,2}^{'}w_{1}^{2'*}, \\
\left. e_{2}^{2'*}\cdot w_{2}^{2'*}+(h^{2'*})^{3}\right).
\end{multline*}
Let us define a graded ring homomorphism $\phi: A^{\bullet}(Z_{2})\rightarrow A^{\bullet}(Z_{2}^{'})$, so
\begin{align*}
& \phi(h^{2*})=a_{0}h^{2'*}+a_{1}e_{1}^{2'*}+a_{2}e_{2}^{2'*}, \\ 
& \phi(e_{1}^{2*})=b_{0}h^{2'*}+b_{1}e_{1}^{2'*}+b_{2}e_{2}^{2'*}, \\
& \phi(e_{2}^{2*})=c_{0}h^{2'*}+c_{1}e_{1}^{2'*}+c_{2}e_{2}^{2'*}, \\
& \phi(w_{1}^{2*})=d_{0}(h^{2'*})^{2}+d_{1}w_{1}^{2'*}+d_{2}w_{2}^{2'*}, \\
& \phi(w_{2}^{2*})=f_{0}(h^{2'*})^{2}+f_{1}w_{1}^{2'*}+f_{2}w_{2}^{2'*}.
\end{align*}
Since $A^{\bullet}(Z_{2}^{'})$ is generated by $\left\{(h^{2'*})^{3}\right\}$ in codimension $3$, then we have that
\begin{equation*}
\phi\left((h^{2*})^{3}\right)=(\phi(h^{2*}))^{3}=(h^{2'*})^{3},
\end{equation*}
so we can conclude that $\phi(h^{2*})=h^{2'*}$, that is, $a_{0}=1, a_{1}=a_{2}=0$. If $\phi$ would be a graded isomorphism, then by \cite[9 Theorem.]{Eick16} the following relations will hold:
\begin{align}
& \phi(h^{2*}\cdot  e_{1}^{2*}-\gamma w_{1}^{2*})=\phi(h^{2*})\cdot\phi(e_{1}^{2*})-\gamma\phi(w_{1}^{2*})=0, \label{EqRel1.4} \\
& \phi(h^{2*}\cdot  e_{2}^{2*}-\eta w_{2}^{2*})=\phi(h^{2*})\cdot\phi(e_{2}^{2*})-\eta\phi(w_{2}^{2*})=0, \label{EqRel2.4} \\
& \phi(h^{2*}\cdot  w_{1}^{2*})=\phi(h^{2*})\cdot\phi(w_{1}^{2*})=0, \label{EqRel3.4} \\
& \phi(h^{2*}\cdot  w_{2}^{2*})=\phi(h^{2*})\cdot\phi(w_{2}^{2*})=0, \label{EqRel4.4} \\
& \phi(e_{1}^{2*}\cdot  e_{2}^{2*}-\beta_{1,2}w_{2}^{2*})=\phi(e_{1}^{2*})\cdot\phi(e_{2}^{2*})-\beta_{1,2}\phi(w_{2}^{2*})=0, \label{EqRel5.4} \\
& \phi(e_{1}^{2*}\cdot w_{2}^{2*})=\phi(e_{1}^{2*})\cdot\phi(w_{2}^{2*})=0, \label{EqRel6.4} \\
& \phi(w_{1}^{2*}\cdot e_{2}^{2*})=\phi(w_{1}^{2*})\cdot\phi(e_{2}^{2*})=0, \label{EqRel7.4} \\
& \phi(w_{1}^{2*}\cdot w_{2}^{2*})=\phi(w_{1}^{2*})\cdot\phi(w_{2}^{2*})=0, \label{EqRel8.4} \\
& \phi(e_{1}^{2*}\cdot w_{1}^{2*}+(h^{2*})^{3})=\phi(e_{1}^{2*})\cdot\phi(w_{1}^{2*})+\phi(h^{2*})^{3}=0, \label{EqRel9.4} \\
& \phi(e_{2}^{2*}\cdot w_{2}^{2*}+(h^{2*})^{3})=\phi(e_{2}^{2*})\cdot\phi(w_{2}^{2*})+\phi(h^{2*})^{3}=0, \label{EqRel10.4} \\
& \phi\left((e_{1}^{2*})^{2}-(4\gamma_{1}-2)w_{1}^{2*}+\gamma_{1}(h^{2*})^{2}\right)=\phi(e_{1}^{2*})^{2}-(4\gamma_{1}-2)\phi(w_{1}^{2*})+\gamma_{1}\phi(h^{2*})^{2}=0. \label{EqRel11.4} \\
& \phi\left((e_{2}^{2*})^{2}-(4\gamma_{2}-2-\beta_{1,2})w_{2}^{2*}+\gamma_{2}(h^{2*})^{2}-\beta_{1,2}w_{1}^{2*}\right)= \nonumber \\
& \phi(e_{2}^{2*})^{2}-(4\gamma_{2}-2-\beta_{1,2})\phi(w_{2}^{2*})+\gamma_{2}\phi(h^{2*})^{2}-\beta_{1,2}\phi(w_{1}^{2*})=0. \label{EqRel12.4}
\end{align}
From Equations (\ref{EqRel1.4}), (\ref{EqRel2.4}), (\ref{EqRel3.4}) and (\ref{EqRel4.4}), we have that: 
\begin{align*}
b_{0} & =d_{0}=0, \\
d_{1} & =\frac{\gamma_{1}^{'}}{\gamma_{1}}b_{1}, \\
d_{2} & =\frac{\gamma_{2}^{'}}{\gamma_{1}}b_{2}, \\
c_{0} & =f_{0}=0, \\
f_{1} & =\frac{\gamma_{1}^{'}}{\gamma_{2}}c_{1}, \\
f_{2} & =\frac{\gamma_{2}^{'}}{\gamma_{2}}c_{2}.
\end{align*}
Moreover, Equations (\ref{EqRel6.4}) and (\ref{EqRel7.4}) implies that $b_{1}f_{1}=-b_{2}f_{2}$. Now, as a consequence of Equations (\ref{EqRel9.4}) and (\ref{EqRel10.4}) the following relations hold $b_{1}d_{1}+b_{2}d_{2}=1$ and $c_{1}f_{1}+c_{2}f_{2}=1$. Finally, it follows from Equations (\ref{EqRel5.4}) and (\ref{EqRel11.4}) that:
\begin{align*}
& b_{1}^{2}(4\gamma_{1}^{'}-2)+b_{2}^{2}\beta_{1,2}^{'}-d_{1}(4\gamma_{1}-2)=0, \\
& 2b_{1}b_{2}\beta_{1,2}^{'}+b_{2}^{2}(4\gamma_{2}^{'}-2-\beta_{1,2}^{'})-d_{2}(4\gamma_{1}-2)=0, \\
& b_{1}c_{1}(4\gamma_{1}^{'}-2)+b_{2}c_{2}\beta_{1,2}^{'}-f_{1}\beta_{1,2}=0, \\
& b_{1}c_{2}\beta_{1,2}^{'}+b_{2}c_{1}\beta_{1,2}^{'}+b_{2}c_{2}(4\gamma_{2}^{'}-2-\beta_{1,2}^{'})-f_{2}\beta_{1,2}=0.
\end{align*}
We can conclude that if $\phi$ would be a graded isomorphism, then $\gamma_{1}=\gamma_{1}^{'}$, $\gamma_{2}=\gamma_{2}^{'}$, $\beta_{1,2}=\beta_{1,2}^{'}$, $b_{1}=d_{1}=c_{2}=f_{2}=1$ and $b_{2}=d_{2}=c_{1}=f_{1}=0$.
\end{ex}

\begin{ex}\label{Ex3}
Let $\mathcal{C}\subset\mathbb{P}^{3}$ be a smooth rational curve of degree $\gamma_{1}$, $P\in\mathcal{C}$ is a closed point and $\pi_{1}: Z_{1}\rightarrow Z_{0}=\mathbb{P}^{3}$ the blow-up of $\mathbb{P}^{3}$ with center $C_{1}=P$. If we denote by $\widetilde{\mathcal{C}}$ to the strict transform of $\mathcal{C}$ $in Z_{1}$, let $\pi_{2}: Z_{2}\rightarrow Z_{1}$ be the blow-up of $Z_{1}$ with center $C_{2}=\widetilde{\mathcal{C}}$, that is $C_{2}\xrightarrow{t} C_{1}$. Then, we know that $A^{\bullet}(Z_{2})$ is generated by
\begin{equation*}
\begin{cases}
\left\{h^{2*},e_{1}^{2*},e_{2}^{2*}\right\} & \text{in codimension}\enspace 1, \\
\left\{(h^{2*})^{2},(e_{1}^{2*})^{2},w_{2}^{2*}\right\} & \text{in codimension}\enspace 2, \\
\left\{(h^{2*})^{3}\right\} & \text{in codimension}\enspace 3,
\end{cases}
\end{equation*}
and it follows from Corollary \ref{CorChowRingSeqRCPBU} that the Chow ring $A^{\bullet}(Z_{2})$ is isomorphic to:
\begin{equation*}
A^{\bullet}(Z_{2})\cong\frac{\mathbb{Z}\left[h^{2*},e_{1}^{2*},e_{2}^{2*},w_{2}^{2*}\right]}{\mathcal{A}},
\end{equation*}
where
\begin{multline*}
\mathcal{A}=\left((h^{2*})^{4}, h^{2*}\cdot e_{1}^{2*}, (e_{1}^{2*})^{3}-(h^{2*})^{3}, (h^{2*})^{2}\cdot e_{2}^{2*}, (e_{1}^{2*})^{2}\cdot e_{2}^{2*}, h^{2*}\cdot e_{1}^{2*}\cdot e_{2}^{2*}, \right. \\
h^{2*}\cdot e_{2}^{2*}-\gamma_{1}w_{2}^{2*}, e_{1}^{2*}\cdot e_{2}^{2*}-w_{2}^{2*}, h^{2*}\cdot w_{2}^{2*}, e_{1}^{2*}\cdot w_{2}^{2*}, (w_{2}^{2*})^{2}, \\
\left. (e_{2}^{2*})^{2}-(4\gamma_{1}-4)w_{2}^{2*}+\gamma_{1}(h^{2*})^{2}+(e_{1}^{2*})^{2}, e_{2}^{2*}\cdot w_{2}^{2*}+(h^{2*})^{3}\right);
\end{multline*} 
Let $\mathcal{C}^{'}\subset\mathbb{P}^{3}$ be a smooth rational curve of degree $\gamma_{1}^{'}$, $Q\in\mathcal{C}^{'}$ is a closed point and $\pi_{1}^{'}: Z_{1}^{'}\rightarrow Z_{0}$ be the blow-up of $\mathbb{P}^{3}$ with center $C_{1}^{'}=\mathcal{C}^{'}$. If we denote by $F$ the inverse image $\pi_{1}^{'-1}(Q)$, let $\pi_{2}^{'}: Z_{2}^{'}\rightarrow Z_{1}^{'}$ the blow-up of $Z_{1}^{'}$ with center $C_{2}^{'}=F$, that is, $C_{2}^{'}\xrightarrow C_{1}^{'}$. Then, we know that $A^{\bullet}(Z_{2}^{'})$ is generated by
\begin{equation*}
\begin{cases}
\left\{h^{2'*},e_{1}^{2'*},e_{2}^{2'*}\right\} & \text{in codimension}\enspace 1, \\
\left\{(h^{2'*})^{2},w_{1}^{2'*},w_{2}^{2'*}\right\} & \text{in codimension}\enspace 2, \\
\left\{(h^{2'*})^{3}\right\} & \text{in codimension}\enspace 3.
\end{cases}
\end{equation*}
and it follows from Corollary \ref{CorChowRingSeqRCPBU} that the Chow ring $A^{\bullet}(Z_{2}^{'})$ is isomorphic to:
\begin{equation*}
A^{\bullet}(Z_{2}^{'})\cong\frac{\mathbb{Z}\left[h^{2'*},e_{1}^{2'*},w_{1}^{2'*},e_{2}^{2'*},w_{2}^{2'*}\right]}{\mathcal{A}^{'}},
\end{equation*}
where
\begin{multline*}
\mathcal{A}^{'}=\left((h^{2'*})^{4}, (h^{2'*})^{2}\cdot e_{1}^{2'*}, h^{2'*}\cdot e_{1}^{2'*}-\gamma_{1}^{'}w_{1}^{2'*}, h^{2'*}\cdot w_{1}^{2'*}, (w_{1}^{2'*})^{2}, \right. \\
(e_{1}^{2'*})^{2}-(4\gamma_{1}^{'}-2)w_{1}^{2'*}+\gamma_{1}^{'}(h^{2'*})^{2}, e_{1}^{2'*}\cdot w_{1}^{2'*}+(h^{2'*})^{3}, h^{2'*}\cdot e_{2}^{2'*}, (e_{1}^{2'*})^{2}\cdot e_{2}^{2'*}, w_{1}^{2'*}\cdot e_{2}^{2'*}, \\
e_{1}^{2'*}\cdot e_{2}^{2'*}-w_{2}^{2'*}, h^{2*}\cdot w_{2}^{2'*}, e_{1}^{2'*}\cdot w_{2}^{2'*}, w_{1}^{2'*}\cdot w_{2}^{2'*}, (w_{2}^{2'*})^{2}, (e_{2}^{2'*})^{2}+w_{2}^{2'*}+w_{1}^{2'*}), \\
\left. e_{2}^{2'*}\cdot w_{2}^{2'*}+(h^{2'*})^{3}\right).
\end{multline*}
Let us define a graded ring homomorphism $\phi: A^{\bullet}(Z_{2})\rightarrow A^{\bullet}(Z_{2}^{'})$, so
\begin{align*}
& \phi(h^{2*})=a_{0}h^{2'*}+a_{1}e_{1}^{2'*}+a_{2}e_{2}^{2'*}, \\ 
& \phi(e_{1}^{2*})=b_{0}h^{2'*}+b_{1}e_{1}^{2'*}+b_{2}e_{2}^{2'*}, \\
& \phi(e_{2}^{2*})=c_{0}h^{2'*}+c_{1}e_{1}^{2'*}+c_{2}e_{2}^{2'*}, \\
& \phi(w_{2}^{2*})=f_{0}(h^{2'*})^{2}+f_{1}w_{1}^{2'*}+f_{2}w_{2}^{2'*}.
\end{align*}
Since $A^{\bullet}(Z_{2}^{'})$ is generated by $\left\{(h^{2'*})^{3}\right\}$ in codimension $3$, then we have that
\begin{equation*}
\phi\left((h^{2*})^{3}\right)=(\phi(h^{2*}))^{3}=(h^{2'*})^{3},
\end{equation*}
so we can conclude that $\phi(h^{2*})=h^{2'*}$, that is, $a_{0}=1, a_{1}=a_{2}=0$. If $\phi$ would be a graded isomorphism, then by \cite[9 Theorem.]{Eick16} the following relations will hold:
\begin{align}
& \phi(h^{2*}\cdot  e_{1}^{2*})=\phi(h^{2*})\cdot\phi(e_{1}^{2*})=0, \label{EqRel1.3} \\
& \phi(h^{2*}\cdot  e_{2}^{2*}-\gamma_{1}w_{2}^{2*})=\phi(h^{2*})\cdot\phi(e_{2}^{2*})-\gamma_{1}\phi(w_{2}^{2*})=0, \label{EqRel2.3} \\
& \phi(h^{2*}\cdot  w_{2}^{2*})=\phi(h^{2*})\cdot\phi(w_{2}^{2*})=0, \label{EqRel3.3} \\
& \phi(e_{1}^{2*}\cdot  e_{2}^{2*}-w_{2}^{2*})=\phi(e_{1}^{2*})\cdot\phi(e_{2}^{2*})-\phi(w_{2}^{2*})=0, \label{EqRel4.3} \\
& \phi(e_{1}^{2*}\cdot w_{2}^{2*})=\phi(e_{1}^{2*})\cdot\phi(w_{2}^{2*})=0, \label{EqRel5.3} \\
& \phi(e_{2}^{2*}\cdot w_{2}^{2*}+(h^{2*})^{3})=\phi(e_{2}^{2*})\cdot\phi(w_{2}^{2*})+\phi(h^{2*})^{3}=0, \label{EqRel6.3} \\
& \phi((e_{1}^{2*})^{3}-(h^{2*})^{3})=\phi(e_{1}^{2*})^{3}-\phi(h^{2*})^{3}=0, \label{EqRel7.3} \\
& \phi\left((e_{2}^{2*})^{2}-(4\gamma_{1}-4)w_{2}^{2*}+\gamma_{1}(h^{2*})^{2}+(e_{1}^{2*})^{2}\right)= \nonumber \\
& \phi(e_{2}^{2*})^{2}-(4\gamma_{1}-4)\phi(w_{2}^{2*})+\gamma_{1}\phi(h^{2*})^{2}+\phi(e_{1}^{2*})^{2}=0. \label{EqRel8.3} 
\end{align}
From Equation (\ref{EqRel1.3}) we have that $b_{0}=b_{1}=0$. Moreover, Equations (\ref{EqRel2.3}) and (\ref{EqRel3.3}) implies that $c_{0}=f_{0}=f_{2}=0$ and $c_{1}\gamma_{1}^{'}=f_{1}\gamma_{1}$. Now, as a consequence of Equations (\ref{EqRel5.3}) and (\ref{EqRel6.3}), the following relation hold:  $c_{1}f_{1}=1$. Finally, it follows from Equations (\ref{EqRel4.3}) and (\ref{EqRel7.3}) that:
\begin{align*}
& b_{2}c_{2}+f_{1}=0, \\
& b_{2}(c_{1}+c_{2})=0, \\
& b_{2}^{3}=1.
\end{align*}
So we can conclude that if $\phi$ would be a graded isomorphism, then $\gamma_{1}=\gamma_{1}^{'}$ and
\begin{enumerate}
\item either $b_{1}=f_{2}=0$, $b_{2}=c_{1}=f_{1}=1$ and $c_{2}=-1$, that is, $\phi(e_{1}^{2*})=e_{2}^{2'*}, \phi(e_{2}^{2*})=e_{1}^{2'*}-e_{2}^{2'*}, \phi(w_{2}^{2*})=w_{1}^{2'*}$, \label{CaseC1}
\item or $b_{1}=f_{2}=0$, $b_{2}=c_{2}=1$ and $c_{1}=f_{1}=-1$, that is, $\phi(e_{1}^{2*})=e_{2}^{2'*}, \phi(e_{2}^{2*})=-e_{1}^{2'*}+e_{2}^{2'*}, \phi(w_{2}^{2*})=-w_{1}^{2'*}$. \label{CaseC2}
\end{enumerate}
In the Case \ref{CaseC1} we have that $\phi^{-1}(e_{1}^{2'*})=e_{1}^{2*}+e_{2}^{2*}$ and $\phi^{-1}(w_{1}^{2'*})=w_{2}^{2*}$, whereas in the Case \ref{CaseC2} it holds that $\phi^{-1}(e_{1}^{2'*})=e_{1}^{2*}-e_{2}^{2*}$ and $\phi^{-1}(w_{1}^{2'*})=-w_{2}^{2*}$. If $\phi^{-1}$ would be a graded isomorphism, then the relation $\phi^{-1}(e_{1}^{2'*})^{2}-(4\gamma_{1}-2)\phi^{-1}(w_{1}^{2'*})+\gamma_{1}\phi^{-1}(h^{2'*})^{2}=0$ will hold, but this is not true in Case \ref{CaseC2}. As a result, conditions $\gamma_{1}=\gamma_{1}^{'}$, $b_{1}=f_{2}=0$, $b_{2}=c_{1}=f_{1}=1$ and $c_{2}=-1$ must hold for $\phi$ to be a graded isomorphism.
\end{ex}

Finally, as an interesting application of the previous results, we determine the type of allowed proximity between two irreducible components of the exceptional divisor when both are regularly and projectively contractable. Firstly we need to define and characterize the notion of final component of the exceptional divisor $E$. The naive idea is that, given a sequence of blow-ups $(Z_{s},...,Z_{0},\pi)$, an irreducible component $E_{i}^{s}$ will be final if there is some other sequence of point blow-ups $(Z_{s}^{'},...,Z_{0}^{'},\pi^{'})$ associated to the sequential morphism $(\pi: Z_{s}\rightarrow Z_{0})$ such that $E_{i}^{s}$ will be the exceptional divisor of the last blow-up of $(Z_{s}^{'},...,Z_{0}^{'},\pi^{'})$.

\begin{defn}\label{Def7}
Let $(Z_{s},\ldots,Z_{0},\pi)$ be a sequence of blow-ups over $k$ as in Definition \ref{DefSeqBU}. The components of the exceptional divisor $E$ in $Z_{s}$ are $\left\{E_{1}^{s},...,E_{s}^{s}\right\}$. Assume that $E_{i}^{s}$ is an irreducible component. Set $E_{i}^{i}$ to be the image of $E_{i}^{s}$ in $Z_{i}$. We say that $E_{i}^{s}$ is final with respect to $(Z_{s},\ldots,Z_{0},\pi)$ if there exists an open set $U_{i}$ on $Z_{i}$ such that $E_{i}^{i}\subset U_{i}$, $V_{i}=\pi_{s,i}^{-1}(U_{i})\subset Z_{s}$, and $\pi_{s,i}\vert_{V_{i}}: V_{i}\rightarrow U_{i}$ is an isomorphism (see Remark \ref{NoPi} for $\pi_{s,i}$).
\end{defn}

\begin{defn}\label{Def8}
Let $\pi: Z_{s}\rightarrow Z_{0}$ be a sequential morphism. We say that an irreducible component $E_{i}^{s}$ of $E$ is final if there exists a sequence of blow-ups $(Z_{s},...,Z_{0},\pi)$ associated to $\pi: Z_{s}\rightarrow Z_{0}$ such that $E_{i}^{s}$ is final with respect to this sequence. 
\end{defn}

\begin{rem}
Notice that in the case of sequences of blow-ups at smooth centers, being final is equivalent to being regularly and projectively contractable for an irreducible component of the exceptional divisor.
\end{rem}

Let $(Z_{s},...,Z_{0},\pi)$ and $(Z_{s}^{'},...,Z_{0}^{'},\pi^{'})$ be two sequences of point and rational curve blow-ups associated to the sequential morphism $\pi: Z_{s}\rightarrow Z_{0}$. Then, there exists an isomorphism $b: Z_{s}\rightarrow Z_{s}^{'}$ that induces an isomorphism between the corresponding Chow rings, $\phi: A^{\bullet}(Z_{s})\rightarrow A^{\bullet}(Z_{s}^{'})$. Let $E_{i}^{s}$ and $E_{j}^{s}$ two final divisors, $E_{i}^{s}$ is final with respect to $(Z_{s},...,Z_{0},\pi)$ and $E_{j}^{s}$ is final with respect to $(Z_{s}^{'},...,Z_{0}^{'},\pi^{'})$, such that $E_{i}^{s}\cap E_{j}^{s}\neq\emptyset$. Then, one of the following proximity configurations is verified:
\begin{enumerate}
\item $E_{i}^{s}\xrightarrow{} E_{j}^{s}$ and $E_{j}^{s}\xrightarrow{} E_{i}^{s}$. In this case, $\phi(e_{i}^{s*})=e_{i}^{s'*}-e_{j}^{s*'}$ and $\phi(e_{j}^{s*})=e_{i}^{s'*}$;
\item $E_{i}^{s}\xrightarrow{t} E_{j}^{s}$ and $E_{j}^{s}\xrightarrow{t} E_{i}^{s}$. In this case, $\phi(e_{i}^{s*})=e_{i}^{s'*}$ and $\phi(e_{j}^{s*})=e_{j}^{s*'}$;
\item $E_{i}^{s}\xrightarrow{t} E_{j}^{s}$ and $E_{j}^{s}\xrightarrow{} E_{i}^{s}$. In this case, $\phi(e_{i}^{s*})=e_{i}^{s'*}-e_{j}^{s*'}$ and $\phi(e_{j}^{s*})=e_{j}^{s'*}$.
\end{enumerate}
Now, next corollary follows from Examples \ref{Ex1}, \ref{Ex2} and \ref{Ex3}.

\begin{cor}
Let $(Z_{s},...,Z_{0},\pi)$ be a sequence of point and rational curve blow-ups, and $E_{i}^{s}$, $E_{j}^{s}$ two final divisors with respect to the sequential morphism $\pi: Z_{s}\rightarrow Z_{0}$ with non-empty intersection, $E_{i}\cap E_{j}\neq\emptyset$. Then $E_{i}^{s}\xrightarrow{t} E_{j}^{s}$ and $E_{j}^{s}\xrightarrow{} E_{i}^{s}$, or viceversa.
\end{cor}

\begin{rem}
The condition $E_{i}^{s}\xrightarrow{t} E_{j}^{s}$ and $E_{j}^{s}\xrightarrow{} E_{i}^{s}$ is also a sufficient condition. In other words, $E_{i}^{s}$ and $E_{j}^{s}$ are both finals with respect to the sequential morphism $\pi: Z_{s}\rightarrow Z_{0}$, with $E_{i}\cap E_{j}\neq\emptyset$, if and only if $E_{i}^{s}\xrightarrow{t} E_{j}^{s}$ and $E_{j}^{s}\xrightarrow{} E_{i}^{s}$, or viceversa. A proof of this result, based on an intensive use of intersection theory can be found in \cite[Theorem 6.1.5]{Camazon25}.
\end{rem}

\subsection{Disclosure statement}

The author does not work for, consult, own shares in or receive funding from any
company or organization that would benefit from this article, and have disclosed
no relevant affiliations beyond their academic appointment.


\subsection*{Acknowledgment}
The author wishes to express his gratitude to Professors A. Campillo and S. Encinas for their constant advice, suggestions and useful discussions during the preparation of this paper.


\end{document}